\documentclass[11pt]{amsart}

\usepackage{amsmath, amsfonts, amssymb, amscd, enumerate}
\usepackage{color}
\usepackage[all]{xy}

\newtheorem{theo}{Theorem}[section]

\newtheorem{example}[theo]{Example}
\newtheorem{definition}[theo]{Definition}

\newtheorem{lemma}[theo]{Lemma}
\newtheorem{theorem}[theo]{Theorem}
\newtheorem{corollary}[theo]{Corollary}
\newtheorem{proposition}[theo]{Proposition}

\renewcommand{\=}{:=}

\newcommand{\beq}{\begin{equation}}
\newcommand{\eeq}{\end{equation}}

\newcommand{\C}{\mathbb{C}}

\renewcommand{\P}{\mathbb{P}}
\renewcommand{\H}{\mathbb{H}}

\newcommand{\HH}{\mathbb{H}}
\newcommand{\HP}{\mathbb{HP}}

\newcommand{\hh}{{\mathbb{H}}}
\newcommand{\cc}{{\mathbb{C}}}
\newcommand{\rr}{{\mathbb{R}}}
\newcommand{\nn}{{\mathbb{N}}}

\newcommand{\sss}{{\mathbb{S}}}

\newcommand{{\ee}}{{\`e}}
\newcommand{{\aA }}{{\`a}}
\newcommand{{\oo}}{{\`o}}
\newcommand{{\uu}}{{\`u}}
\newcommand{{\ii}}{{\`i}}
\newcommand{{\BL}}{{Bl_0(\HH^n)}}
\newcommand{{\BLHP}}{{Bl_{[1,0,0]}(\HP^2)}}
\newcommand{\bl}{{Bl_0(B^n)}}

\renewcommand{\square}{\kern1pt\vbox
{\hrule height 0.6pt\hbox{\vrule width 0.6pt\hskip 3pt
\vbox{\vskip 6pt}\hskip 3pt\vrule width 0.6pt}\hrule height0.6pt}\kern1pt}

\def\<#1,#2>{\langle\,#1,\,#2\,\rangle}
\newcommand{\arr}{\begin{array}{rlll}}
\newcommand{\ea}{\end{array}}
\newcommand{\bea}{\begin{eqnarray}}
\newcommand{\eea}{\end{eqnarray}}
\newcommand{\bean}{\begin{eqnarray*}}
\newcommand{\eean}{\end{eqnarray*}}

\def\sideremark#1{\ifvmode\leavevmode\fi\vadjust{
\vbox to0pt{\hbox to 0pt{\hskip\hsize\hskip1em
\vbox{\hsize3cm\tiny\raggedright\pretolerance10000
\noindent #1\hfill}\hss}\vbox to8pt{\vfil}\vss}}}

\newcounter{ssig}
\newcounter{ttig}
\begin{document}
\author{Graziano Gentili}
\address{Dipartimento di Matematica e Informatica ``U. Dini'', Universit\`a di Firenze, 50134 Firenze, Italy}
\email{gentili@math.unifi.it}



\author{Anna Gori}
\address{Dipartimento di Matematica, Universit\`a di Milano, Via Saldini 50, 20133 Milano, Italy}
\email{anna.gori@unimi.it}

\author{Giulia Sarfatti}
\address{Istituto Nazionale di Alta Matematica ``F. Severi'', Citt\`a Universitaria, Piazzale Aldo Moro 5, 00185 Roma, Italy \and Institut de Math\'ematiques de Jussieu,
Universit\'e Pierre et Marie Curie, 4, place Jussieu, F-75252 Paris, France}
\email{sarfatti@math.unifi.it}
\thanks{\rm This project has been supported by G.N.S.A.G.A. of INdAM - Rome (Italy), by MIUR of the Italian Government (Research Projects: PRIN ``Real and complex manifolds: geometry, topology and harmonic analysis'' and FIRB ``Geometric function theory and differential geometry").}

\title{A direct approach to quaternionic manifolds}

\begin{abstract}{The recent definition of slice regular function of several quaternionic variables suggests a new notion of  quaternionic manifold. We give the definition of quaternionic regular manifold, as a space locally modeled on $\H^n$, in a slice regular sense. We exhibit some significant classes of examples, including manifolds which carry a quaternionic affine structure.}\\
{\bf keywords:} Regular functions of  quaternionic variables, quaternionic manifolds
\end{abstract}
\maketitle

\section{Introduction}
The definition of {\em slice regular}  function of one quaternionic variable,  given in \cite{GS},  has originated  a rich area of  research that is producing a good quaternionic counterpart of the theory of holomorphic functions. In this theory - whose status was recently presented in \cite{libroGSS} -  many results hold in analogy with the complex case, while  several properties show how deep the differences between the complex and the quaternionic setting can be. Slice regular functions play a fundamental role in surprising applications: one of them appears in differential geometry, where slice regular functions can be used for the classification of Orthogonal Complex Structures in certain subdomains of the space of quaternions, \cite{GenSalSto}. 

Let $\H$ denote the skew field of quaternions and let $\sss\subset \H$ denote the $2$-dimensional sphere of purely imaginary quaternions, $\sss=\{q\in \H : q^2=-1\}$. Then $\H$ is decomposed in ``slices" as follows 
\[
\H =\bigcup_{I\in \sss} \rr+\rr I
\]
where each $L_I=\rr+\rr I$ is isomorphic to the complex plane $\cc$. Slice regularity in one variable is then defined as follows (see \cite{libroGSS}).
\begin{definition}
Let $\Omega$ be a domain in $\hh$ and let $f : \Omega \to \hh$ be a function. For all $I \in \sss$, let us denote $L_I = \rr + \rr I$, $\Omega_I = \Omega \cap L_I$ and $f_I = f_{|_{\Omega_I}}$. 
The function $f$ is called \emph{(slice) regular} if, for all $I \in \sss$, the restriction $f_I$ is holomorphic, i.e. the function $\bar \partial_I f : \Omega_I \to \hh$ defined by
$$
\bar \partial_I f  = \frac{1}{2} \left( \frac{\partial}{\partial x}+I\frac{\partial}{\partial y} \right) f_I 
$$
vanishes identically on $\Omega_I$.
\end{definition}
\noindent The class of slice regular functions shares with the one-Century-old family of Fueter regular functions (see \cite{fueter}) the problem of not being closed under composition, but it has many nice properties: probably the most important is the fact of including the natural polynomials and power series of a quaternionic variable, \cite{sferiche}.

The interest and the richness of the theory of holomorphic functions of several complex variables immediately encourages the natural search for a quaternionic analog. Such a search is also urged by the desire to define and study quaternionic manifolds. It turns out that, mainly due to the lack of commutativity, slice regular functions of several quaternionic variables are not easy to define. After different attempts, a recent and satisfactory approach  is based on  the use of
{\em stem functions} (see Definition \ref{stemfunction}) introduced in \cite{fueter} by Fueter himself. Indeed Fueter, after he gave his  definition of quaternionic regularity in 1934, \cite{fueter},  envisaged a possible different approach. It  was only very recently, with  the paper by Ghiloni and Perotti \cite{GhiloniPerotti}, that this approach was fully understood and used to define slice regular functions of several quaternionic variables (see Definition \ref{regularseveral}). It is worthwhile noticing that, in one quaternionic variable, the definition of slice regularity made with the use of stem functions is  a generalization of the direct definition that we have presented above (see \cite{GP}).

How to give a ``direct'' definition of quaternionic manifold is an old, well motivated, problem addressed by many authors. Indeed, several different definitions  of quaternionic manifolds, as spaces locally modeled on $\H^n$, have been given, and have produced interesting theories. We like to mention here, in particular, the results by Kulkarni, \cite{kulkarni}, and Sommese, \cite{So}. Needless to say, the recently defined class of slice regular functions of several quaternionic variables encourges a further attempt to give a definition of quaternionic manifold. It is well known that the main difficulty that one encounters in doing such an attempt is the fact that in general  slice regularity is not preserved by composition. For instance, in the one variable case, 
the only classes of slice regular functions preserved by compositions are the class of affine functions $q\mapsto qa+b$ with $a,b\in \H$,  and the so called  {\em slice preserving} functions, i.e. those slice regular functions $f:\Omega\to \H$ such that $f(\Omega_I)\subseteq \Omega_I$, for all $I\in \sss$. 
As a consequence,  natural examples of quaternionic regular manifolds in dimension one are:  the quaternionic tori studied in \cite{Bis-gen} and the Hopf quaternionic manifolds considered in \cite{So}, whose transition functions are quaternionic affine;  the classical quaternionic projective space $\H\P^1$, whose transition functions are slice preserving.

The aim of this paper is to give a definition of (slice) quaternionic regular manifold in higher dimension and provide some significant classes of examples. 
 \begin{definition}
A differentiable manifold $M$ of $4n$ real dimensions is a \emph{quaternionic regular manifold} if it admits a differentiable atlas $\{(U_i, \varphi_i )\}_i$, $\varphi_i:U_i\to \HH^n$, whose transition functions 
\[\varphi_{ij}:=\varphi_j\circ \varphi_i^{-1}:\HH^n \supseteq \varphi_i(U_i\cap U_j)\to \varphi_j(U_i\cap U_j)\subseteq \HH^n,  \]
\[\varphi_{ij}:(q_1,\ldots , q_n) \mapsto (\varphi^1_{ij}(q_1, \ldots, q_n), \ldots, \varphi^n_{ij}(q_1, \ldots, q_n) )\]
are such that each component $\varphi^k_{ij}: \varphi_i(U_i\cap U_j) \to \HH$ is a slice regular function of $n$ quaternionic variables.     
\end{definition}

\noindent It turns out that Sommese's quaternionic manifolds, \cite{So}, including the   quaternionic Iwasawa manifold, are quaternionic regular.

After presenting some preliminary results, we  begin by proving, providing explicit  coordinates,  that, for all $n\in \nn$, the quaternionic projective space $\H\P^n$ and  the blow-up $Bl_0(\H^n)$ of $\H^n$ at $0$ are quaternionic regular manifolds.

We then identify an important class of quaternionic regular manifolds that, following what Kobayashi did in the complex case, \cite{Kobayashi},  we call \emph{affine quaternionic manifolds}, or \emph{manifolds with an affine quaternionic structure}.  Examples of these  manifolds are constructed as quotients of $\H^n$ with respect to the action of  subgroups of affine transformatons of $\H^n$ that act freely and properly discontinuously. An exhaustive study of this class of manifolds will be the object of a forthcoming paper. Paper \cite{alberto} by D\'iaz,Verjovsky and Vlacci is strictly related to this subject.

Then we pass to prove a few surgery-type results, concerning the connected sum of a quaternionic regular manifold $M$ with $\HP^n$, and the blow-up of $M$ at a point $p\in M$. Namely, we prove the following result. 
\begin{theorem}
\label{sum1}If $M$ is a regular quaternionic manifold of quaternionic dimension $n$, then the connected sum of $M$ and $\HP^n$ is quaternionic diffeomorphic to the blow-up of $M$ at a point $p\in M$. Moreover, both manifolds are quaternionic regular.
\end{theorem}

We conclude the paper by noticing that the Grassmannians of $p$-dimensional $\H$-planes in $\H^n$ are not quaternionic regular manifolds with respect to their natural atlases.

\section{Preliminaries on regular functions of several variables}

We report here for the convenience of the reader the definition of regular function of several quaternionic variables.
This notion was introduced by Ghiloni and Perotti in \cite{GhiloniPerotti}.
Let $\rr_n$ denote the real Clifford algebra of signature $(0,n)$ generated by $e_1,\ldots, e_n$. Each element of $\rr_n$ is of the form $x=\sum_{K\in \mathcal{P}(n)}e_K x_K$ where $\mathcal{P}(n)$ is the powerset of $\{1,\ldots ,n \}$, $e_K=e_{k_1}\cdots e_{k_s}$  (with ${k_1}<k_2<\cdots <{k_s}$) are the basis element of $\rr_n$ (for $K=\emptyset$ one gets the unit of $\rr_n$) and $x_K$ are real numbers.
\begin{definition}\label{stemfunction}
Let $D$ be an open subset of $\C^n$ invariant with respect to complex conjugation in each variable $z_1,\ldots, z_n$.  
A continuous function $F:D \to \HH \otimes_{\rr} \rr_n$ of the form $F=\sum_{K\in \mathcal{P}(n)}e_K F_K$ is called a {\em stem function} if it is Clifford intrinsic, i.e. for any $K\in \mathcal{P}(n)$, $h\in \{1,\ldots, n\} $ and $z=(z_1,\ldots, z_n)\in D$ the components $F_K:D \to \HH$ satisfy
\[ F_K(z_1,\ldots,z_{h-1}, \bar z_h, z_{h+1}, \ldots, z_n)=\left\{\begin{array}{c l}F_K(z) & \text{if } h\notin K,\\
-F_K(z) & \text{if } h \in K. 
\end{array} \right.\]
\end{definition}
Each stem function defined on $D\subseteq \C^n$ (open subset invariant by separate conjugation) induces a slice function defined on the {\em circular} subset of $\HH^n$ associated with $D$,
\[\Omega_D=\{(q_1,\ldots, q_n)\in \HH^n \ | \ q_h=x_h+y_hJ_h, J_h\in \mathbb{S}, \forall h=1,...,n; (x_1+ y_1i, \ldots, x_n+y_ni)\in D \}.\]
\begin{definition}
Let $F:D \to \HH \otimes_{\rr} \rr_n$, with $F=\sum_{K\in \mathcal{P}(n)}e_K F_K$, be a stem function. We define the {\em left slice function} $\mathcal{I}(f):\Omega_D \to \HH$ induced by $F$ as 
\[\mathcal{I}(F)(x_1+y_1J_1, \ldots, x_n+y_nJ_n):=\sum_{K\in \mathcal{P}(n)}J_K F_K(x_1+y_1i, \ldots, x_n+y_ni)\]
where 
\[J_K=\prod_{k\in K}^{\to}J_k=J_{k_1}\cdots J_{k_s}\]
is the ordered product.
\end{definition}


In order to give the definition of  {\em slice regular} 
functions of several quaternionic variables, we need to give a notion of holomorphicity for stem functions.

\begin{definition} 
For any $h=1, \ldots, n$ define the {\em complex structure} $\mathcal{J}_h$ on $\rr_n$ as
\[\mathcal{J}_h(e_K):=\left\{\begin{array}{c r}
-e_{K \setminus \{h\}} & \text{if } h\in K,\\
e_{K \cup \{h\}} & \text{if } h\notin K,
\end{array} \right.\] 
\end{definition}
\noindent Notice that $\mathcal{J}_h^2=-id_{\rr_n}$, for any $h=1, \ldots, n$ so that $\mathcal{J}_h$ defines an (almost) complex structure on $\rr_n$. 
\noindent It is possible to extend any almost complex structure $\mathcal{J}_h$ to $\HH\otimes_{\rr}\rr_n$ by setting
\[\mathcal{J}_h(a \otimes x)=a \otimes \mathcal{J}_h(x)\]
for any $a\in \HH$, $x\in \rr_n$.
\begin{definition}\label{regularseveral}
Let $D$ be an open subset of $\C^n$ invariant with respect to complex conjugation in each variable, let $F:D \to \HH \otimes_{\rr} \rr_n$ be a stem function of class $C^1$ and let $f=\mathcal{I}(F):\Omega_D\to \HH$ be the induced slice function. $F$ is called a {\em holomorphic stem function} if for any $h=1, \ldots, n$ and any fixed $z^0=(z_1^0,\ldots,z_n^0)\in D$, the function
\[F_h^{z^0}:D_h\to (\HH \otimes_{\rr}\rr_n, \mathcal{J}_h), \ \  z_h\mapsto F(z_1^0,\ldots, z_{h-1}^0, z_h,z_{h+1}^0,\ldots,z_n^0)\]
is holomorphic on a domain $D_h$ of $\C$ containing $z_h$. 
Equivalently, if 
\[\overline{\partial}_h F:= \frac{1}{2}\left(\frac{\partial F}{\partial x_h}+\mathcal{J}_h\frac{\partial F}{\partial y_h}\right)= 0\]
on $D$ for every $h=1,\ldots, n$.

If $F$ is holomorphic, the induced function $f=\mathcal{I}(F)$ is called a {\em left slice regular function} on $\Omega_D$. 
\end{definition}

It is possible to give a ``symmetric'' definition, namely the definition of {\em right} slice regular function.  
\begin{definition}
Let $F:D \to \HH \otimes_{\rr} \rr_n$, with $F=\sum_{K\in \mathcal{P}(n)}e_K F_K$, be a stem function. We define the {\em right slice function} $\mathcal{I}^r(F):\Omega_D \to \HH$ induced by $F$ as 
\[\mathcal{I}^r(F)(x_1+y_1J_1, \ldots, x_n+y_nJ_n):=\sum_{K\in \mathcal{P}(n)} F_K(x_1+y_1i, \ldots, x_n+y_ni)J_K\]
where 
\[J_K=\prod_{k\in K}^{\to}J_k=J_{k_1}\cdots J_{k_s}\]
is the ordered product.
\end{definition}

Right slice regularity can be expressed in terms of the family of (almost) complex structures (extended as before to $\HH \otimes_{\rr}\rr_n$) defined as follows.
\begin{definition} 
For any $h=1, \ldots, n$ define the {\em complex structure} $\mathcal{J}^r_h$ on $\rr_n$ as
\[\mathcal{J}^r_h(e_K):=-(\mathcal{J}_h(e_K^c))^c,\]
where $x\mapsto x^c$ denotes the Clifford conjugation in $\rr_n$. 
\end{definition}
\begin{definition}\label{peroni}
Let $D$ be an open subset of $\C^n$ invariant with respect to complex conjugation in each variable, let $F:D \to \HH \otimes_{\rr} \rr_n$ be a stem function of class $C^1$ and let $f=\mathcal{I}^r(F):\Omega_D\to \HH$ be the induced right slice function. $F$ is called a {\em right holomorphic stem function} if 
\[\overline{\partial}^r_h F:= \frac{1}{2}\left(\frac{\partial F}{\partial x_h}+\mathcal{J}^r_h\frac{\partial F}{\partial y_h}\right)= 0\]
on $D$ for every $h=1,\ldots, n$.

If $F$ is right holomorphic, the induced right slice function $f=\mathcal{I}^r(F)$ is called a {\em right slice regular function} on $\Omega_D$. 
\end{definition} 

\begin{example}
Let us  collect here some significant examples, referring to \cite{GhiloniPerotti} for an extensive description. We  point out that the ordering of the  variables is important for regularity:
\begin{itemize} 
\item the function $(q_1,q_2)\mapsto q_1^{-1}q_2$ is  left slice regular on $\H\setminus \{0\}\times \H$;
\item the function $(q_1,q_2)\mapsto q_2q_1$ is  right slice regular on $\H\times \H$;
\item the function $(q_1,q_2,q_2)\mapsto q_2q_1q_3$ is  neither right nor left slice regular.
\end{itemize}
\end{example}
\begin{definition} \label{nostra} Let $f=(f_1,f_2,\ldots,f_m):\Omega_D\to\H^m$ be a differentiable function. We say that $f$ is {\em slice regular}  if, and only if, each component $f_i$ is either  a right or a  left slice regular function of $n$ quaternionic variables.  
\end{definition}

\section{Quaternionic regular manifolds}

\noindent In this section, we give the announced definition of quaternionic regular manifold and we exhibit several examples.

\begin{definition}
A differentiable manifold $M$ of $4n$ real dimensions is a {\em quaternionic regular manifold} if it admits a differentiable atlas $\{(U_i, \varphi_i )\}_i$, $\varphi_i:U_i\to \HH^n$, whose transition functions 
\[\varphi_{ij}:=\varphi_j\circ \varphi_i^{-1}:\HH^n \supseteq \varphi_i(U_i\cap U_j)\to \varphi_j(U_i\cap U_j)\subseteq \HH^n,  \]
\[\varphi_{ij}:(q_1,\ldots , q_n) \mapsto (\varphi^1_{ij}(q_1, \ldots, q_n), \ldots, \varphi^n_{ij}(q_1, \ldots, q_n) )\]
are 
slice regular functions of $n$ quaternionic variables.    Such an atlas is called  a {\em regular atlas.}
\end{definition}

\begin{definition} 
Let $M$ and $N$ be two quaternionic regular manifolds with  regular atlases  $\{(U_i, \varphi_i )\}_i$ and $\{(V_j, \psi_j )\}_j$ respectively.
A differentiable function $f:M\to N$  is called  {\em quaternionic regular}  at a point $p\in M$ if, and only if,  there exist two systems of local coordinates $\{(U_{i_0}, \varphi_{i_0} )\}$ for $p$, and $\{(V_{j_0}, \psi_{j_0} )\}$ for $f(p)$, so that $\psi_{j_0}\circ f\circ{\varphi_{i_0}}^{-1}$ is slice regular. The function $f$ is called {\em quaternionic regular} if it is quaternionic regular at all $p\in M$. 
\end{definition}
\noindent Notice that the quaternionic regular functions between manifolds, that appear in the examples considered in this paper, have the property that  their expressions are slice regular  in all local coordinates of the given regular atlases.  However one cannot expect that this holds in general, since composition of slice regular functions does not always maintain regularity.

\subsection{Quaternionic projective spaces}

Of course $\H^n$ is a quaternionic regular manifold.
If $(q_1,\ldots,q_{n+1})\in \H^{n+1}\setminus \{0\},$ then $[q_1,\ldots,q_{n+1}]$ denotes the (right) vector line $\{(q_1\lambda,\ldots,q_{n+1}\lambda)\in \H^{n+1}:\lambda\in \H\}$ of $\H^{n+1}$. As usual $\H\P^n$ denotes the set of (right) vector lines in $\H^{n+1}$.
 It is easy to see that  the natural system of coordinates of a quaternionic projective space endows $\H\P^n$ with a structure of quaternionic regular manifold: the usual transition functions are easily seen to be slice regular.
\subsection{The blow-up of $\hh^n$}
In order to prove that the blow-up of $\H^n$ is quaternionic regular, we use
 the natural construction of a structure of differentiable manifold on the blow-up of $\HH^n$ at a point. To begin with, consider the projection onto the quotient space
\begin{equation*}
\varphi: \HH^n\setminus \{0\} \to \HP^{n-1}
\end{equation*}
defined, with classical notation, by
\begin{equation*}
\varphi(q_1,\ldots ,q_n)=[q_1,\ldots ,q_n],
\end{equation*}
and  set $\Gamma_\varphi \subset \HH^n\times \HP^{n-1}$ to be the graph of the map $\varphi$, i.e. the set
\begin{equation*}
\Gamma_\varphi =\{((q_1,\ldots ,q_n), [a_1,\ldots ,a_n] ) \in (\HH^n\setminus\{0\}) \times \HP^{n-1} : (q_1,\ldots ,q_n)\in [a_1,\ldots ,a_n] \}.
\end{equation*}

\begin{definition}\label{blow-up}
The subset $Bl_0(\HH^n)=\Gamma_\varphi \cup (\{0\}\times \HP^{n-1})$ of the Cartesian product $\HH^n\times \HP^{n-1}$, endowed with the natural induced topology, is called the \emph{blow-up of $\HH^n$ at $0$}. The subset $\{0\}\times \HP^{n-1}$ is called the exceptional set of $\BL$.
\end{definition}

\noindent The blow-up of $\HH^n$ at $0$ is a Hausdorff, paracompact and connected topological space, that can be described as follows:

\begin{proposition}\label{blow-up_parole}
The blow-up $\BL$ of $\HH^n$ at $0$ can be globally described as the set 
\begin{equation*}
A=\{((q_1,\ldots ,q_n), [a_1,\ldots ,a_n] ) \in \HH^n \times \HP^{n-1} : (q_1,\ldots ,q_n) \ \textnormal{is in the line} \ [a_1,\ldots ,a_n] \}.
\end{equation*}
\end{proposition}
\begin{proof}
The point $0\in \HH^n$ belongs to all quaternionic $1$-dimensional subspaces of  $\HH^n$, and hence $(\{0\}\times \HP^{n-1})\in A$. Moreover, for $(q_1,\ldots ,q_n)\neq 0$, we have that $(q_1,\ldots ,q_n)$ belongs to the quaternionic $1$-dimensional subspace $[q_1,\ldots ,q_n]$ of $\HP^{n-1}$. As a consequence $\Gamma_\varphi$ is also contained in $A$. There are no other possible elements of $A$, which consequently coincides with $\BL$.
\end{proof}
We now state a couple of technical lemmas, which use the notion of Dieudonn\'e determinant of $2\times 2$ matrices with quaternionic entries. As it is known, the \emph{Dieudonn\'e determinant}, defined as
\begin{equation}\label{Dieudonne}
 det_\HH 
\begin{pmatrix}
a&b\\
c&d
\end{pmatrix}
=\sqrt{|a|^2|d|^2+|c|^2|b|^2-Re(c\bar ab\bar d)}
\end{equation}
generalizes the notion of  determinant to the case of quaternionic matrices (see, e.g. \cite{Bisi-Gentili}). 
Notice that a straightforward extension of the usual definition of determinant to matrices with quaternionic entries does not maintain the key properties that one expects from the determinant. For example, for $a, c, \lambda$ quaternions,  the two columns of the matrix 
\[
\begin{pmatrix}
a&a\lambda\\
c&c\lambda
\end{pmatrix}
\]
are $\HH-$linearly dependent, but the Cayley determinant  of the same matrix (see e.g. \cite{As})  $ac\lambda -ca\lambda$ does not vanish when $ac\neq ca$ and $\lambda\neq 0$. Moreover, if $ac\neq ca$, then for $\lambda=c^{-1}a^{-1}ca\neq 1$ we have that the two columns of the matrix 
\[
\begin{pmatrix}
a&a\\
c&c\lambda
\end{pmatrix}
\]
are $\HH-$linearly independent but $ac\lambda-ca=0$.

 \begin{lemma}\label{proprieta} The Dieudonn\'e determinant $det_\HH$ has the following properties:
 For any $\lambda, \mu 
    \in \mathbb{H}$ and any matrix $X=\begin{pmatrix}

             x & y \\
             z & t 
             \end{pmatrix} \in M(2, \mathbb{H})$ we have that
\begin{itemize}	\vskip .2cm    
	     \item[i)] $
	     det_{\mathbb{H}} \begin{pmatrix}
             x & y\lambda \\
             z & t\lambda \\
             \end{pmatrix}= 
	     det_{\mathbb{H}}\begin{pmatrix}
             x\lambda  & y \\
             z\lambda  & t \\
             \end{pmatrix} =|\lambda|det_{\mathbb{H}}
	     \begin{pmatrix}
             x & y \\
             z & t\\
             \end{pmatrix}
	     $ \vskip .4cm
\item[ii)] $
det_{\mathbb{H}}\begin{pmatrix}
             \mu x & \mu y \\
             z & t \\
             \end{pmatrix}= 
	     det_{\mathbb{H}}\begin{pmatrix}
             x  & y \\
             \mu z  & \mu t \\
             \end{pmatrix} =|\mu|det_{\mathbb{H}}
	     \begin{pmatrix}
             x & y \\
             z & t\\
             \end{pmatrix}
	    $ \vskip .3cm
 \item[iii)] If the matrix $Y$ is obtained from the matrix $X$ by: (a)
substituting to a row the sum of the two rows, or (b) substituting to
a column the sum of the two columns, then
$det_{\mathbb{H}}(X)=det_{\mathbb{H}}(Y).$
\end{itemize}
 \end{lemma}
 \begin{proof} A direct substitution and computation show the
     assertions (see also \cite{Bisi-Gentili}).
     \end{proof}
\noindent It is natural at this point to search for the inverse of a $2\times 2$ quaternionic matrix having nonvanishing Dieudonn\'e determinant (see \cite{Bisi-Gentili}). 
%
%
\noindent If  the matrix 
 $A= \left[ \begin{array}{ll}

             a & b \\
             c & d \\
             \end{array} \right] \in M(2, \mathbb{H})
$ 
is right-invertible, then $a\neq 0$ or $b\neq 0$. If $a\neq 0$, then the right inverse of $A$ is
\begin{equation}\label{inversa a}
  \left[ \begin{array}{ll}
             x & y \\
             t & z \\
             \end{array} \right] =
               \left[ \begin{array}{ll}
            a^{-1} + a^{-1}b(d-ca^{-1}b)^{-1}ca^{-1} &
-a^{-1}b(d-ca^{-1}b)^{-1} \\
             -(d-ca^{-1}b)^{-1}ca^{-1} & (d-ca^{-1}b)^{-1} \\
             \end{array} \right].
\end{equation}
If instead we are in the case $b\ne 0$ then, the right inverse has the form
\begin{equation}\label{inversa b}
  \left[ \begin{array}{ll}
             x & y \\
             t & z \\
             \end{array} \right] =
               \left[ \begin{array}{ll}
            -(c-db^{-1}a)^{-1}db^{-1} & (c-db^{-1}a)^{-1}\\
            b^{-1}+b^{-1}a(c-db^{-1}a)^{-1}db^{-1} &
-b^{-1}a(c-db^{-1}a)^{-1} \\
             \end{array} \right].
\end{equation}
\noindent As one may expect, when $ab\ne 0$ then the two forms
    (\ref{inversa a}) and (\ref{inversa b}) of the inverse of $A$ do
coincide. If $abcd\neq 0$ then the inverse matrix of $A$ assumes an
even nicer
form,
\begin{equation}\label{bilatera}
  \left[ \begin{array}{ll}
             x & y \\
             t & z \\
             \end{array} \right] =
               \left[ \begin{array}{ll}
            (a-bd^{-1}c)^{-1} & (c-db^{-1}a)^{-1} \\
            (b-ac^{-1}d)^{-1} &  (d-ca^{-1}b)^{-1} \\
             \end{array} \right]
\end{equation}
which allows a Cramer-type rule to solve $2 \times 2$ linear systems
with quaternionic coefficients (see again \cite{Bisi-Gentili}).

\begin{lemma}\label{proporzionali}
Let $(q_1,\ldots ,q_n)$ and $(a_1,\ldots ,a_n)$ be elements of $\HH^n$, with $(a_1,\ldots ,a_n) \neq 0$. Then there exists $\lambda\in \HH$ such that 
$(q_1,\ldots ,q_n)=(a_1,\ldots ,a_n)\lambda$ if, and only if, the following system of $n \choose 2$ equations is satisfied:
\begin{equation*}
det_\HH 
\begin{pmatrix}
q_i&a_i\\
q_j&a_j
\end{pmatrix}=0
\end{equation*}
for all $1\leq i<j\leq n$.
\end{lemma}
\begin{proof}
The proof is a consequence of Lemma \ref{proprieta}. In fact if there is $\lambda\in \HH$ such that $(a_1,\ldots ,a_n)\lambda=(q_1,\ldots ,q_n)$, then by Lemma \ref{proprieta} $iii)$, for all $1\leq i<j\leq n$, we get 
\begin{equation*}
det_\HH 
\begin{pmatrix}
q_i&a_i\\
q_j&a_j
\end{pmatrix}= 
det_\HH \begin{pmatrix}
a_i\lambda&a_i\\
a_j\lambda&a_j
\end{pmatrix}=|\lambda|det_\HH \begin{pmatrix}
a_i&a_i\\
a_j&a_j
\end{pmatrix}=|\lambda|det_\HH \begin{pmatrix}
0&a_i\\
0&a_j
\end{pmatrix}=0.
\end{equation*}
Conversely, if  for all $1\leq i<j\leq n$,
\begin{equation*}
det_\HH 
\begin{pmatrix}
q_i&a_i\\
q_j&a_j
\end{pmatrix}=0,
\end{equation*}
then in particular (re-ordering the entries if necessary) we have $a_1\neq 0$ and
\begin{equation}\label{n-equa}
det_\HH 
\begin{pmatrix}
q_1&a_1\\
q_j&a_j
\end{pmatrix}=0
\end{equation}
for all $1<j\leq n$. If $\lambda\in \HH$ is such that $q_1=-a_1\lambda$, then again by Lemma \ref{proprieta}
\begin{eqnarray}\label{n-equa1}
0=det_\HH 
\begin{pmatrix}
q_1&a_1\\
q_j&a_j
\end{pmatrix}=
det_\HH 
\begin{pmatrix}
q_1&a_1\lambda\\
q_j&a_j\lambda
\end{pmatrix} \\
=det_\HH 
\begin{pmatrix}
0&a_1\lambda\\
q_j+a_j\lambda&a_j\lambda
\end{pmatrix}=\sqrt{|a_1\lambda|^2|q_j+a_j\lambda|^2}
\end{eqnarray}
and hence either $\lambda=0$ (and hence $q_1=0$) and $q_j=0$, or $q_j+a_j\lambda=0$ (for all $1<j\leq n$). In any case $q_j=-a_j\lambda$ for all $1<j\leq n$. We then conclude that $(q_1,\ldots ,q_n)=(a_1,\ldots ,a_n)\lambda$. 
\end{proof}

We can now prove

\begin{theorem}\label{DifferentiableManifold}
The blow-up $\BL$ of $\HH^n$ at $0$ is a  quaternionic regular manifold of real dimension $4n$. Moreover the exceptional set $\{0\}\times \HP^{n-1}$ is a quaternionic  submanifold of real dimension $4n-4$.
\end{theorem}
\begin{proof}

To begin with notice that the two natural projections 
\begin{eqnarray*}
\pi_1 &:& \BL \to \HH^n \qquad \pi_1((q_1,\ldots ,q_n), [a_1,\ldots ,a_n] )= (q_1,\ldots ,q_n)\\
\pi_2 &: &\BL \to \HP^{n-1} \qquad \pi_2((q_1,\ldots ,q_n), [a_1,\ldots ,a_n] )= [a_1,\ldots ,a_n]
\end{eqnarray*}
are continuous maps from $\BL$ to $\HH^n$ and $\HP^{n-1}$, respectively. If we use homogeneous coordinates $[a]=[a_1,\ldots,a_n]$  in  $\HP^{n-1}$ and set $q=(q_1,\ldots ,q_n)\in \HH^n$ then, using the Dieudonn\'e determinant \eqref{Dieudonne}, we can write
\begin{equation*}
\BL=\{(q, [a]) \in  \HH^n \times \HP^{n-1} : |q_i|^2|a_j|^2+|q_j|^2|a_i|^2-Re(q_j\bar q_ia_i\bar a_j)=0, \tiny{1\le i<j\le n} \}.
\end{equation*}
On $\HP^{n-1}$ we use the $n$ natural systems of coordinates $\left\{ (U_i, p_i)\right\}_{i=1}^n$ where $U_i=\{[a] \in \HP^{n-1} : a_i\neq 0\}$ and where  $p_i:U_i \to \HH^{n-1}$ is such that  
\begin{equation}
p_i([a_1,\ldots,a_{i-1}, a_i, a_{i+1}, \ldots a_n])=(a_1a_i^{-1},\ldots, a_{i-1}a_i^{-1},a_{i+1}a_i^{-1}, \ldots, a_{n}a_i^{-1})
\end{equation}
In turn, a $4n$ dimensional, differentiable atlas $\left\{ (V_i, \varphi_i)\right\}_{i=1}^n$ can be constructed on $\BL$ by considering the $n$ open sets $V_i=\pi_2^{-1}(U_i)$ that cover all of $ \BL$ and, for each $i=1,\ldots,n,$ the homeomorphism of $V_i$ onto $\HH^n$ defined by
\begin{equation*}
\varphi_i(q,[a])= (a_1a_i^{-1},\ldots, a_{i-1}a_i^{-1},q_i,a_{i+1}a_i^{-1}, \ldots, a_{n}a_i^{-1})
\end{equation*}
whose inverse is 
\begin{equation*}
\varphi_i^{-1}(b_1,\ldots,b_n)= ((b_1b_i,\ldots, b_{i-1}b_i,b_i,b_{i+1}b_i, \ldots, b_{n}b_i) , [(b_1,\ldots, b_{i-1},1,b_{i+1}, \ldots, b_{n}])
\end{equation*}
At this point, for $1\leq i\neq j\leq n $ we can easily compute the change of coordinates
\begin{eqnarray*}
\varphi_j\circ \varphi_i^{-1}(b_1,\ldots,b_n)= \varphi_j((b_1b_i,\ldots, b_{i-1}b_i,b_i,b_{i+1}b_i, \ldots, b_{n}b_i) , [(b_1,\ldots, b_{i-1},1,b_{i+1}, \ldots, b_{n}])\\=
(b_1b_j^{-1},\ldots,b_{i-1}b_j^{-1},b_j^{-1},b_{i+1}b_j^{-1},\ldots,b_{j-1}b_j^{-1},b_jb_i,b_{j+1}b_j^{-1},\ldots,b_nb_j^{-1})
\end{eqnarray*}
that is clearly a diffeomorphism. To prove that $\pi_1$ is differentiable it is enough to compute explicitly the map $\pi_1\circ \varphi_i^{-1}: \HH^n \to \HH^n$
\begin{eqnarray}\label{cambio}
\pi_1\circ \varphi_i^{-1}(b_1,\ldots, b_n)=(b_1b_i,\ldots, b_{i-1}b_i,b_i,b_{i+1}b_i, \ldots, b_{n}b_i)
\end{eqnarray}
and see that this map is differentiable  for every $i=1,\ldots,n$. Analogously, the fact that the map $p_i \circ \pi_2\circ \varphi_i^{-1}: \HH^n \to \HH^{n-1}$ 
\begin{eqnarray*}
p_i\circ \pi_2\circ \varphi_i^{-1}(b_1,\ldots, b_n)=(b_1,\ldots, b_{i-1},b_{i+1}, \ldots, b_{n})
\end{eqnarray*}
is differentiable for every $i=1,\ldots,n$ proves that $\pi_2$ is differentiable.  The transition functions $\varphi_j\circ \varphi_i^{-1}$ have components which are right or left slice regular, thus showing that $Bl_0(\H^n)$ is a quaternionic regular manifold.

To prove that the exceptional set is a quaternionic regular submanifold of real dimension $4n-4$, we can notice that, for $i=1,\ldots,n$, the intersection $(\{0\}\times \HP^{n-1})\cap V_i$  is mapped by $\varphi_i$ to the $\HH-$hyperplane $b_i=0$, which has real dimension $4n-4$. 
\end{proof}


As in the  case of $\BL$, we can define the blow-up $Bl_0(B^n)$ of the closed unit ball $B^n=\overline {\mathbb{B}}^n$ of $\HH^n$ at the point $0$. 
The restriction to $\bl$ of the charts of the atlas of $\BL$ define on $\bl$ a natural structure of regular quaternionic manifold (with boundary).

\subsection{Quaternionic affine manifolds} The Dieudonn\'e determinant, that we have encountered in dimension 2, has been defined for any $n\times n$ quaternionic matrix $A\in M(n,\H)$ (see, e.g. \cite{dieu}). It can be used, in the usual fashion,  to define the group of quaternionic $n\times n$ invertible matrices $GL(n, \H)$. For $Q=(q_1,q_2,\ldots,q_n)\in \H^n$,  we can define the group of all quaternionic affine transformations 
$$
\mathcal A(n,\H)=\{Q \mapsto QA+B : A\in GL(n,\H), B=(b_1, b_2, \ldots , b_n)\in \H^n \}
$$ 
which is included in the class of slice regular functions. In complete analogy with what Kobayashi does in the complex case, \cite{Kobayashi}, we give the following: 

\begin{definition}
A differentiable manifold $M$ of $4n$ real dimensions has  a \emph{quaternionic affine structure} if it admits a differentiable atlas whose transition functions are restrictions of quaternionic affine functions of $\mathcal A(n, \H)$.     
\end{definition}

\noindent In particular, differentiable manifolds endowed with a  quaternionic affine structure are quaternionic regular. This can be used to construct  a large class of quaternionic regular manifolds; indeed, for any subgroup $\Gamma \subset \mathcal A(n, \H)$ which acts freely and properly discontinuously on $\H^n$, the quotient space
\[
M=\H^n\slash \Gamma
\]
admits an atlas whose transition functions are slice regular belonging to $\Gamma \subset \mathcal A(n,\H)$, and hence has  a quaternionic affine structure. It is worthwhile noticing that the \emph{quaternionic manifolds} studied by Sommese in \cite{So} all admit a quaternionic affine structure, and hence many significant examples can be found in his paper.

%

\subsection{The connected sum of a quaternionic regular manifold $M$ and $\HP^n$}
In what follows, the elements of the space $\HH^{n+1}$ are denoted by $(w, q_1,\ldots,q_n)$, and we set $|q|^2=q_1\bar  q_1+\cdots+q_n\bar q_n=|q_1|^2+\cdots+|q_n|^2$. We consider the quaternionic regular manifold $\HP^{n}$,  and the map 
$$h: (\HP^{n} \setminus {[1,0,\ldots, 0]}) \to \HH^{n}$$ 
defined by 
\begin{equation*}
h([w, q_1, \ldots, q_{n}] )=
(\frac{q_1}{|q|^{2}} \bar w, \ldots, \frac{q_n}{|q|^{2}}\bar w)
\end{equation*}
Let us notice that $h$ is well defined, since 
\begin{eqnarray*}
&&h([w\lambda, q_1\lambda, \ldots, q_{n}\lambda] ) = (\frac{q_1\lambda}{|q|^{2}|\lambda|^2} \bar \lambda \bar w, \ldots, \frac{q_n\lambda}{|q|^{2}|\lambda|^2}\bar \lambda \bar  w)\\
&&= (\frac{q_1}{|q|^{2}} \bar w, \ldots, \frac{q_n}{|q|^{2}}\bar w)=h([w, q_1, \ldots, q_{n}] ),
\end{eqnarray*}
for all non zero quaternions $\lambda$. It is straightforward to prove that $h$ is differentiable, and we have 
\begin{proposition}\label{diffeo1}
The map $H: (\HP^{n} \setminus {[1,0,\ldots, 0]}) \to \BL$ defined by 
\begin{eqnarray}
H([w, q_1, \ldots, q_{n}] )=((\frac{q_1}{|q|^{2}} \bar w, \ldots, \frac{q_n}{|q|^{2}}\bar w), [\frac{q_1}{|q|^{2}}, \ldots, \frac{q_n}{|q|^{2}}])
\end{eqnarray}
is a diffeomorphism.
\end{proposition}
\begin{proof}
Since there exists $i\in \{1,\ldots, n\}$ such that $q_i\neq 0$, then  we can use the charts $U_{i+1}$ and $V_{i}$, (see the proof of Theorem \ref{DifferentiableManifold}) and write the map $H$ as
\begin{eqnarray*}
&&(w, q_1, \ldots, q_{i-1}, q_{i+1}, \ldots, q_n)\mapsto \\ &&\mapsto (w, q_1, \ldots, q_{i-1}, 1, q_{i+1}, \ldots, q_n) \mapsto (q_{1}, \ldots, q_{i-1}, \frac{\bar w}{|q|^{2}}, q_{i+1},\ldots, q_n).
\end{eqnarray*}
As a consequence, the map $H$ is differentiable. The inverse map 
$$H^{-1}:\BL \to (\HP^{n} \setminus {[1,0,\ldots, 0]})$$ 
can be expressed as
\begin{eqnarray*}
H^{-1}((b_1u, \ldots, b_nu), [b_1,\ldots,b_n])=[\bar u, \frac{b_1}{|b|^2},\ldots, \frac{b_n}{|b|^2}]
\end{eqnarray*}
and, in coordinates on $\varphi_{i}(V_{i})$, $i=1,\ldots,n$,
\begin{eqnarray*}
&&(b_1,\ldots,b_{i-1}, u , b_{i+1},\ldots, b_n) \mapsto \\
&&\mapsto ((b_1u,\ldots,b_{i-1}u,u,b_{i+1}u,\ldots, b_nu),[ b_1,\ldots,b_{i-1},1,b_{i+1},\ldots, b_n])\\
&&\mapsto [\bar u,\frac{b_1}{|b|^2},\ldots, \frac{b_{i-1}}{|b|^2}, \frac{1}{|b|^2}, \frac{b_{i+1}}{|b|^2}, \ldots, \frac{b_n}{|b|^2}]\\
&&\mapsto (\bar u|b|^2, b_1,\ldots, b_n).
\end{eqnarray*}
Therefore $H^{-1}$ is differentiable and $H$  is a diffeomorphism.
\end{proof}
\noindent By simply restricting the map $H$, and with obvious notations, we obtain the following

\begin{corollary}\label{diffeo}
The restriction of the map 
$$
H([w, q_1, \ldots, q_{n}] )=((\frac{q_1}{|q|^{2}} \bar w, \ldots, \frac{q_n}{|q|^{2}}\bar w), [\frac{q_1}{|q|^{2}}, \ldots, \frac{q_n}{|q|^{2}}])
$$
to $(\HP^{n} \setminus {[1\times \mathbb{B}^n]})$ establishes a diffeomorphism between $(\HP^{n} \setminus {[1\times \mathbb{B}^n]})$ and $\bl$.
\end{corollary}
\noindent In analogy with the real and complex cases, one can define the blow-up $Bl_p(M)$ of a quaternionic regular  manifold $M$ of real dimension $4n,$ at a point $p$, as the manifold obtained  from $M$ by substituting  the point $p$ with the set of all quaternionic (right) vector lines of the tangent space $T_p M$ (i.e. with the $\H\P^{n-1}$ obtained as the quaternionic projective space over $T_p M\cong \H^n$). After recalling that blowing-up at a point is a local operation we are  able to prove that

\begin{theorem}\label{previous}
Let $M$ be a real $4n$-dimensional differentiable manifold. The connected sum of $M$ and $\HP^n$ is diffeomorphic to the blow-up of $M$ at a point $p\in M$.
\end{theorem}
\begin{proof}
As it is well known, to construct the connected sum described in the statement we can start by considering, on one side, a chart $(U, \varphi)$ of $M$ such that $\varphi(p)=0$ and that $\varphi(U) \supset \mathbb{B}^n$. In this way $\varphi^{-1}|_{\partial \mathbb{B}^n}:\partial \mathbb B^n \to \varphi^{-1}(\partial \mathbb B^n)$ is the ``glueing diffeomorphism'' of $M\setminus \varphi^{-1}(\mathbb B^n)$ and $B^n$ along the boundaries $\varphi^{-1}(\partial \mathbb B^n)$ and $\partial \mathbb B^n$. We need now to consider $(\HP^{n} \setminus {[1\times \mathbb B^n]})$, which - by Corollary \ref{diffeo}  - is a manifold with boundary diffeomorphic to $\bl$. The projection map $\pi_1$ defined in Theorem \ref{DifferentiableManifold} induces a diffeomorphism between $(\bl\setminus \{0\}\times \HP^{n-1})$ and $B^n\setminus\{0\}$, whose restriction 
$$\pi_1|_{\partial \bl}: \partial \bl \to \partial \mathbb B^n$$
can be used to construct the glueing diffeomorphism 
$$\varphi^{-1}\circ \pi_1|_{\partial \bl}: \partial \bl \to \varphi^{-1}(\partial \mathbb B^n)$$ of the connected sum $M\#\  \HP^n$. 
\end{proof}

\subsection{The blow-up of a quaternonic regular manifold at a point} It is easy to prove the following consequence of Theorem \ref{previous}.
\begin{theorem}
\label{sum}If $M$ is a regular quaternionic manifold of quaternionic dimension $n$, then the connected sum of $M$ and $\HP^n$ is quaternionic diffeomorphic to the blow-up of $M$ at a point $p\in M$. Moreover, both manifolds are quaternionic regular.
\end{theorem}
\begin{proof}
It is enough to show that gluing the boundary of $\bl$ with that of $M\setminus \varphi^{-1}(\mathbb B^n)$ we obtain a regular quaternionic manifold. To this aim, we can find  an open neighborhood  $V$ such that $Bl_0((1+\epsilon)B^n)\setminus \{{0}\times\HP^{n-1}\} \subseteq V
\subseteq \BL \setminus \{{{0}\times \HP^{n-1}}\}$ and consider the map $\psi:V \rightarrow M$ defined as in the preceding proof by $\varphi^{-1}\circ \pi_1$. 
Indeed, using the chart $\varphi_i, (i=1,\ldots,n),$ of $Bl_0(\H^n)$, see \eqref{cambio}, and the chart $\varphi$ of $M$ we obtain that $\varphi\circ\psi\circ \varphi_i^{-1}(q_1,\ldots,q_n)=\varphi\circ\varphi^{-1}\circ \pi_1\circ \varphi_i^{-1}(q_1,\ldots,q_n)=(q_1q_i,\ldots, q_{i-1}q_i,q_i,q_{i+1}q_i, \ldots, q_{n}q_i)$ is a regular diffeomorphism between open sets of $\H^n$. Therefore  $(\varphi^{-1}\circ \pi_1\circ \varphi_i)^{-1}$ is a compatible chart for the quaternionic regular atlas of the connected sum of $M$ and $\HP^n$. The assertion follows.
\end{proof}

\noindent For the sake of completeness, we recall here that the issue of orientation is important when using the connected sum; for example, when considering the sum of two copies of $\HP^n$, one of the copies has to be endowed with the reverse orientation with respect to the other. For this reason, to keep trace of the different orientation of the two copies, the connected sum is classically denoted by $\HP^n \#\ \overline{ \HP^n}$.

\begin{corollary}\label{somma}
The connected sum $\HP^n \#\  \overline{\HP^n}$ is quaternionic diffeomorphic to the blow-up of $\HP^n$ at a point $p\in \HP^n$. \end{corollary}

\subsection{Quaternionic Grassmannians are not quaternionic regular manifolds} \textnormal{A natural generalization of the quaternionic projective space $\HP^n$, $n\in \nn$, is the quaternionic Grassmannian $Gr_{n,p}(\H)$, $p<n\in \nn$. One may expect that the natural transition functions that one uses to define the quaternionic Grassmannian  of $p-$planes in $\H^n$ are slice regular functions. This is not the case: in fact even the Grassmannian $Gr_{2,4}(\H)$ of $2-$planes in $\H^4$ is not a quaternionic regular manifold with respect to the atlas of natural charts that one constructs. To see this, consider the following transition function obtained when the Dieudonn\'e determinant of  the matrix}  $A= \left[ \begin{array}{ll}

             a & b \\
             c & d \\
             \end{array} \right] \in M(2, \mathbb{H})
$ 
\textnormal{is non zero and $abcd\neq 0$ (see \eqref{bilatera}):}
\begin{eqnarray*}
 \left[ \begin{array}{ll}

             a & b \\
             c & d \\
             \end{array} \right] \mapsto 
             \left[ \begin{array}{ll}

             1 & 0\\
             0 & 1\\
             a & b \\
             c & d \\
             \end{array} \right] \mapsto
 \left[ \begin{array}{ll}

             1 & 0\\
             0 & 1\\
             a & b \\
             c & d \\
             \end{array} \right] 
 \left[ \begin{array}{ll}

             a & b \\
             c & d \\
             \end{array} \right]^{-1}= \\
            = \left[ \begin{array}{cc}
             
             (a-bd^{-1}c)^{-1} & (c-db^{-1}a)^{-1} \\
            (b-ac^{-1}d)^{-1} &  (d-ca^{-1}b)^{-1} \\
             1 & 0\\
             0 & 1\\
             \end{array} \right] \mapsto \left[ \begin{array}{cc}
             
             (a-bd^{-1}c)^{-1} & (c-db^{-1}a)^{-1} \\
            (b-ac^{-1}d)^{-1} &  (d-ca^{-1}b)^{-1} \\
             \end{array} \right].
\end{eqnarray*}
\textnormal{In view of Definition \ref{nostra}, the function}
\[
 \left[ \begin{array}{ll}

             a & b \\
             c & d \\
             \end{array} \right] \mapsto \left[ \begin{array}{cc}
             
             (a-bd^{-1}c)^{-1} & (c-db^{-1}a)^{-1} \\
            (b-ac^{-1}d)^{-1} &  (d-ca^{-1}b)^{-1} \\
             \end{array} \right]
\]
\textnormal{is not slice regular.}

\subsection*{Acknowledgments} The authors  wish to thank Alessandro Perotti for useful discussions.

\end{document}